\documentclass[11pt, a4paper]{article}

\usepackage{fullpage}
\usepackage[margin=2.5cm]{geometry}
\usepackage{setspace}
\usepackage[section]{placeins}
\usepackage{authblk}

\usepackage[round]{natbib}

\usepackage{xcolor}
\usepackage[colorlinks=true, allcolors= blue]{hyperref}
\hypersetup{
    bookmarksdepth=3 %
}

\usepackage{amsmath,amssymb,amsthm,mathtools,mathrsfs}
\usepackage[capitalize]{cleveref}

\theoremstyle{plain}
\newtheorem{theorem}{Theorem}
\newtheorem{conjecture}{Conjecture}
\newtheorem{lemma}{Lemma}
\newtheorem{corollary}{Corollary}

\theoremstyle{definition}

\newtheorem{example}{Example}

\newcommand{\Q}{\mathbb{Q}}
\newcommand{\R}{\mathbb{R}}
\newcommand{\N}{\mathbb{N}}
\newcommand{\E}{\mathbb{E}}
\newcommand{\PR}{\mathbb{P}}
\DeclareMathOperator{\val}{val}
\newcommand{\PP}{\R[z]_{<m}}
\newcommand{\PQ}{\Q[z]_{<m}}
\newcommand{\defas}{\coloneqq}
\newcommand{\I}{\mathcal{I}}
\newcommand{\J}{\mathcal{J}}
\newcommand{\C}{\mathcal{C}}

\long\def\ACKNOWLEDGMENT#1{\section*{Acknowledgments}#1}

\allowdisplaybreaks

\title{Values of Absorbing Recursive Games Express\\ All Real Algebraic Numbers}

\author[1]{Ali Asadi}
\author[1]{Krishnendu Chatterjee}

\affil[1]{Institute of Science and Technology Austria (ISTA), Klosterneuburg, Austria}

\date{\today}

\begin{document}
\maketitle

\begin{abstract}
    Many classes of two-player zero-sum stochastic games have the orderfield property: if all payoffs and transition probabilities lie in a subfield of $\R$, so does the undiscounted value.
Absorbing games fail this property, and Oliu-Barton and Vigeral [\emph{Absorbing games with irrational values}, Oper.\ Res.\ Lett.\ 51 (2023) 555--559] conjectured the precise extent of the failure:
every real algebraic number $\alpha$ of degree $m\geq 1$ over $\Q$ is the undiscounted value of a rational $m\times m$ absorbing game.
We prove this conjecture, and in fact within a special subclass of absorbing games: for every such $\alpha$, the game realizing it is \emph{strictly absorbing} and \emph{recursive}, i.e., every action pair is absorbing with positive probability and all non-absorbing stage payoffs are zero; when $\alpha>0$ it can moreover be taken \emph{positive recursive}, with positive absorbing payoffs. As a corollary, the set of undiscounted values of rational $m\times m$ absorbing games is exactly the set of real algebraic numbers of degree at most $m$.\\


\end{abstract}


\newpage

\section{Introduction}\label{sec:intro}

\citet{Sha53} introduced zero-sum stochastic games. Two players repeatedly play a zero-sum game whose data depend on a state variable, which evolves stochastically under the influence of both players' actions. For each discount factor $\lambda\in(0,1]$, the $\lambda$-discounted game has a value $v_\lambda$, and both players have optimal stationary strategies \citep{Sha53}. \citet{BK76} proved that $v_\lambda$ converges as $\lambda\to 0$; the limit $v\defas \lim_{\lambda\to 0}v_\lambda$ is the \emph{undiscounted value}, and by \citet{MN81} it coincides with the \emph{uniform value} of the game.

A finite zero-sum matrix game has its value in any subfield of $\R$ containing its payoffs; rational payoffs, in particular, yield a rational value. Whether this \emph{orderfield property} extends to stochastic games is a recurring theme in their algorithmic and algebraic theory: for many subclasses, if all payoffs and transition probabilities lie in a subfield $K\subseteq\R$, then so does the undiscounted value. For general stochastic games it fails, since with rational data the undiscounted value is algebraic but possibly irrational \citep{BK76}.

\emph{Absorbing games}, formalized by \citet{Koh74}, are stochastic games in which all states but one are absorbing; the celebrated Big Match \citep{Gil57,BF68} belongs to this class. Whether they retain the orderfield property was settled negatively only recently by \citet{OV23} (see also the earlier preprint \citealp{OB23arxiv}), who showed that there exist rational $3\times 3$ absorbing games, with deterministic transitions, whose undiscounted value is irrational; that every real quadratic algebraic number is the undiscounted value of a rational $2\times 2$ absorbing game; and, for every $m\geq 1$, that some rational $m\times m$ absorbing game has an undiscounted value of algebraic degree exactly $m$. This last statement shows that the upper bound of \citet{OB21}, namely that the undiscounted value of a rational $m\times m$ absorbing game is algebraic of degree at most $m$, is tight. These results led \citet{OV23} to the following conjecture.

\begin{conjecture}[{\citealp[Conjecture~1]{OV23}}]\label{conj:OV}
Any real algebraic number of degree $m\geq 1$ is the undiscounted value of a rational $m\times m$ absorbing game.
\end{conjecture}

The cases $m=1$ and $m=2$ \citep[Theorem~3]{OV23} were known.
We prove the conjecture in full, and moreover realize every value within a well-studied restricted subclass of absorbing games. 
Recall that an absorbing game is \emph{strictly absorbing} (also called \emph{halting}) if every action pair is absorbing with positive probability, i.e., under any pair of strategies play reaches an absorbing state with probability $1$~\citep{Sha53}. It is \emph{recursive} if every non-absorbing stage payoff is $0$, i.e., the payoff accrues only upon absorption, and \emph{positive recursive} if, in addition, every absorbing payoff is positive~\citep{Eve57}.
In contrast to general absorbing games where history-dependent strategies may be necessary, e.g., the classical Big Match \citep{Gil57,BF68}, halting games admit optimal stationary strategies \citep{Sha53}, while recursive games admit $\epsilon$-optimal stationary strategies \citep{Eve57}.
Positive recursive games, in particular, are closely related to the concurrent reachability and safety games studied in the verification literature~\citep{AHK98}.
Our main result is as follows.

\begin{theorem}\label{thm:main-intro}
Let $m\geq 1$ and let $\alpha\in\R$ be an algebraic number of degree $m$ over $\Q$. Then there exists a rational $m\times m$ halting and recursive absorbing game whose undiscounted value equals $\alpha$. If $\alpha>0$, then the game is moreover positive recursive.
\end{theorem}

Since duplicating actions of both players leaves all discounted values, and hence the undiscounted value, unchanged, \cref{thm:main-intro} realizes every real algebraic number of degree at most $m$ as the value of a rational $m\times m$ absorbing game. Combining this with the degree bound of \citet{OB21} yields an exact characterization.

\begin{corollary}\label{cor:char-intro}
For every $m\geq 1$, the set of undiscounted values of rational $m\times m$ absorbing games is exactly the set of real algebraic numbers of degree at most $m$ over $\Q$.
\end{corollary}


\paragraph{Strategy of the proof.}
Recall that an absorbing game is given by three real matrices, indexed by the players' actions $(i,j)$: the stage payoff $g=(g_{ij})$, the absorption probability $q=(q_{ij})$ with $q_{ij}\in[0,1]$, and the absorbing payoff $w=(w_{ij})$. When the players choose actions $i$ and $j$, the payoff $g_{ij}$ is received and, with probability $q_{ij}$, the game is absorbed, all subsequent payoffs then equalling $w_{ij}$. The proof combines two ingredients.

\emph{(1) Reduction to a single matrix-game equation.}
For a strictly absorbing game, i.e., one with $q_{ij}>0$ for all $(i,j)$, the asymptotic analysis reduces to one scalar equation: the discounted values $v_\lambda$ converge as $\lambda\to0$ to the unique zero $z^*$ of the map $z\mapsto\val W_0(z)$, where
\[
W_0(z)\defas\bigl(q_{ij}(w_{ij}-z)\bigr)_{i,j}
\]
is the limit auxiliary matrix of \citet{AO19} and \citet{OV23}, and $\val$ denotes the value of a matrix game (\cref{lem:uniformlimit}); in this limit the non-absorbing stage payoffs $g_{ij}$ play no role. Two features of this characterization are essential. First, it is an equation for the value itself, not merely a polynomial equation that the value satisfies, so the ``root-selection'' difficulty that obstructs determinant-based approaches to \cref{conj:OV} (see the discussion in \citealp[Section~5]{OV23}) does not arise. Second, to certify $\val W_0(\alpha)=0$ it suffices to exhibit mixed actions $x^*,y^*$ with $x^{*\top}W_0(\alpha)\geq0$ and $W_0(\alpha)\,y^*\leq0$ componentwise; our construction produces such $x^*,y^*$ with equalities throughout.

\emph{(2) An algebraic construction.}
By~(1), it suffices to find, for a real algebraic number $\alpha$ of degree $m$ with minimal polynomial $\mu$, rational $m\times m$ matrices $A$ and $Q$ with $Q$ entrywise positive and $\val(A-\alpha Q)=0$. From such a pair one builds the game: fixing an integer $N\geq\max_{ij}Q_{ij}$ and setting $g_{ij}\defas0$, $q_{ij}\defas Q_{ij}/N$, and $w_{ij}\defas A_{ij}/Q_{ij}$ yields a strictly absorbing game with $W_0(z)=\tfrac1N(A-zQ)$, which is moreover positive recursive when $\alpha>0$ (\cref{thm:main}). We construct $A$ and $Q$ so that $A-\alpha Q$ has strictly positive left and right null vectors $\tilde x$ and $\tilde y$,
\[
\tilde x^{\top}(A-\alpha Q)=0,
\qquad
(A-\alpha Q)\,\tilde y=0,
\qquad
\tilde x,\tilde y>0 ;
\]
normalizing $\tilde x$ and $\tilde y$ to probability vectors gives optimal strategies that secure exactly $0$ in the matrix game $A-\alpha Q$, hence $\val(A-\alpha Q)=0$.
These matrices come from the arithmetic of the field $\Q[z]/(\mu)$: $A$ and $Q$ represent multiplication by suitable polynomials, in a rational basis of the space of polynomials of degree at most $m-1$. The difficulty is to adapt this basis to $\alpha$ such that the two null vectors, formed from the values of the basis polynomials at $\alpha$ and from the Lagrange interpolation polynomial of $\alpha$, are strictly positive. A perturbation and density argument provides such a basis (\cref{sec:algebra}).

Finally, it is noteworthy that the resulting games are completely explicit. Fully worked examples, for $\sqrt[3]3$ and for the real root of $z^3-z-1$, are given in \cref{sec:example}.

\section{Preliminaries}\label{sec:prelim}

We present the class of zero-sum absorbing games, their dynamics, and the discounted and undiscounted values; see \citep{Koh74,OV23}.

\paragraph{Notation.}
Calligraphic letters (e.g., $\I,\J$) denote sets, their elements (e.g., $i,j$) appear in lowercase, and random elements use uppercase (e.g., $I,J$). 
For a finite set $\C$, let $\Delta(\C)$ be the set of probability distributions over $\C$. 
The sets of real, rational, natural, and non-zero natural numbers are denoted by $\R$, $\Q$, $\N$, and $\N^*$, respectively. 
For a positive integer $m$, let $[m] \defas \{1,\ldots,m\}$.
For a matrix $M\in\R^{\I\times\J}$, $M^\top$ denotes its transpose, and $\|M\|_\infty\defas\max_{i,j}|M_{ij}|$.
Moreover, the notation $M>0$ means that every entry of $M$ is strictly positive.

\paragraph{Model.}
A \emph{(zero-sum) absorbing game}, denoted by $\Gamma$, is defined by a tuple $\Gamma=(\I,\J,g,q,w)$, where:
\begin{itemize}
\item $\I=[m]$ is the finite set of actions of Player~1;
\item $\J=[n]$ is the finite set of actions of Player~2;
\item $g\colon\I\times\J\to\R$ is the (non-absorbing) stage payoff function;
\item $q\colon\I\times\J\to[0,1]$ is the absorption probability function; and 
\item $w\colon\I\times\J\to\R$ is the absorbing payoff function.
\end{itemize}
We write $g_{ij}\defas g(i,j)$, $q_{ij}\defas q(i,j)$, and $w_{ij}\defas w(i,j)$, identify $g$, $q$, and $w$ with matrices in $\R^{\I\times\J}$, and refer to $\Gamma$ as an $m\times n$ absorbing game. The absorbing payoff $w_{ij}$ is relevant only when $q_{ij}>0$.

\paragraph{Subclasses.}
The game $\Gamma$ is \emph{rational} if $g_{ij},q_{ij},w_{ij}\in\Q$ for every $(i,j)\in\I\times\J$. It is \emph{strictly absorbing} (also called \emph{halting}) if $q_{ij}>0$ for every $(i,j)\in\I\times\J$, i.e., every action pair is absorbing with positive probability.
It is \emph{recursive} if $g_{ij}=0$ for every $(i,j)\in\I\times\J$, i.e., payoff accrues only upon absorption, and \emph{positive recursive} if, in addition, $w_{ij}>0$ whenever $q_{ij}>0$.

\paragraph{Dynamics.}
A play of $\Gamma$ proceeds in stages. At each stage $\ell\in\N^*$, provided no absorption has yet occurred:
\begin{enumerate}
\item Player~1 and Player~2 select, simultaneously and independently, actions $I_\ell\in\I$ and $J_\ell\in\J$;
\item a stage payoff $G_\ell\defas g(I_\ell,J_\ell)$ is produced;
\item absorption occurs with probability $q(I_\ell,J_\ell)$.
\end{enumerate}
Once absorption occurs at some stage $\ell_0$, the game is \emph{absorbed}: every subsequent stage payoff equals the absorbing payoff, i.e., $G_\ell\defas w(I_{\ell_0},J_{\ell_0})$ for all $\ell>\ell_0$.

\paragraph{History.}
A \emph{history} before stage $\ell$ is a finite sequence $h=(i_1,j_1,\dots,i_{\ell-1},j_{\ell-1})$ of action pairs, with the empty history $\varnothing$ before stage~$1$. Let $\mathcal H_\ell\defas(\I\times\J)^{\ell-1}$ be the set of histories before stage $\ell$, where $(\I\times\J)^0\defas\{\varnothing\}$, and let $\mathcal H\defas\bigcup_{\ell\geq1}\mathcal H_\ell$.

\paragraph{Play.}
A \emph{play} is an infinite sequence of action pairs, i.e., an element of $\Omega\defas(\I\times\J)^{\N^*}$.

\paragraph{Strategy.}
A \emph{(behavior) strategy} of Player~1 is a map $\sigma\colon\mathcal H\to\Delta(\I)$ assigning a mixed action $\sigma(h)\in\Delta(\I)$ to every history $h\in\mathcal H$. 
The set of such strategies is denoted by $\Sigma$. 
Similarly, a strategy of Player~2 is a map $\tau\colon\mathcal H\to\Delta(\J)$, and $\mathcal T$ denotes the corresponding set. 
A strategy is \emph{stationary} if it plays the same mixed action at every stage and is denoted by an element $x\in\Delta(\I)$ (resp.\ $y\in\Delta(\J)$).

\paragraph{Probability Measure.}
Given a pair of strategies $(\sigma,\tau)\in\Sigma\times\mathcal T$, the dynamics above induce, by the Ionescu--Tulcea theorem, a unique probability measure $\PR_{\sigma,\tau}$ over the set of plays $\Omega$. We denote by $\E_{\sigma,\tau}$ the corresponding expectation. 

\paragraph{Payoff.}
Fix a strategy pair $(\sigma,\tau)\in\Sigma\times\mathcal T$. For a discount factor $\lambda\in(0,1]$, the \emph{$\lambda$-discounted payoff} is
\[
\gamma_\lambda(\sigma,\tau)\defas\E_{\sigma,\tau}\Bigl[\sum_{\ell\geq1}\lambda(1-\lambda)^{\ell-1}G_\ell\Bigr],
\]
and the \emph{long-run average payoff} is the limit inferior of the average stage payoffs:
\[
\gamma(\sigma,\tau)\defas\E_{\sigma,\tau}\Bigl[\liminf_{L\to\infty}\frac1L\sum_{\ell=1}^{L}G_\ell\Bigr].
\]
In both cases Player~1 maximizes and Player~2 minimizes.

\paragraph{Value.}
For every discount factor $\lambda\in(0,1]$, the $\lambda$-discounted game has a \emph{value} $v_\lambda\in\R$, and both players have optimal stationary strategies \citep{Sha53}:
\begin{equation}\label{eq:minimax}
v_\lambda=\max_{x\in\Delta(\I)}\ \min_{y\in\Delta(\J)}\ \gamma_\lambda(x,y)
=\min_{y\in\Delta(\J)}\ \max_{x\in\Delta(\I)}\ \gamma_\lambda(x,y),
\end{equation}
the outer extrema being attained at optimal stationary strategies.
Similarly, the undiscounted game has an \emph{undiscounted value}
\begin{equation}\label{eq:limitvalue}
v\defas\sup_{\sigma\in\Sigma}\ \inf_{\tau\in\mathcal T}\ \gamma(\sigma,\tau)
=\inf_{\tau\in\mathcal T}\ \sup_{\sigma\in\Sigma}\ \gamma(\sigma,\tau)\ \in\R,
\end{equation}
the two sides coinciding by the results of \citet{Koh74,MN81}. Moreover, $v=\lim_{\lambda\to0^+}v_\lambda$.

\section{Background on Absorbing Games}\label{sec:background}

We recall the facts on which our construction relies: the value of a matrix game and its elementary properties, and the auxiliary matrices of \citet{AO19} and \citet{OV23} that link an absorbing game to a one-parameter family of matrix games.

\subsection{Matrix Games}

A zero-sum matrix game is a matrix $M\in\R^{\I\times\J}$ in which Player~1 (the maximizer) selects a row $i\in\I$, Player~2 (the minimizer) selects a column $j\in\J$, and the payoff to Player~1 is $M_{ij}$.
By von Neumann's minimax theorem \citep{vN28}, this game has a \emph{value}
\[
\val(M)\defas\max_{x\in\Delta(\I)}\min_{y\in\Delta(\J)}x^{\top}My=\min_{y\in\Delta(\J)}\max_{x\in\Delta(\I)}x^{\top}My,
\]
attained at optimal mixed actions $x^*\in\Delta(\I)$ and $y^*\in\Delta(\J)$.

The following elementary properties of the value are used repeatedly.

\begin{lemma}\label{lem:matrixgames}
Let $M,M'\in\R^{\I\times\J}$, let $c\in\R$ and $t\geq0$. Then:
\begin{enumerate}
\item[\emph{(i)}] $\val(M+c\,\mathbf{1}\mathbf{1}^{\top})=\val(M)+c$, where $\mathbf{1}\mathbf{1}^{\top}$ is the all-ones matrix;
\item[\emph{(ii)}] if $M'_{ij}\leq M_{ij}-c$ for all $(i,j)$, then $\val(M')\leq\val(M)-c$;
\item[\emph{(iii)}] $|\val(M)-\val(M')|\leq\|M-M'\|_\infty$;
\item[\emph{(iv)}] $\val(tM)=t\,\val(M)$; and 
\item[\emph{(v)}] if there exists $x\in\Delta(\I)$ with $x^{\top}M\geq0$ componentwise, then $\val(M)\geq0$; if there exists $y\in\Delta(\J)$ with $My\leq0$ componentwise, then $\val(M)\leq0$. In particular, if both vectors exist then $\val(M)=0$.
\end{enumerate}
\end{lemma}

\begin{proof}
(i) and (iv) are immediate from the minimax formula. (ii): for every $x\in\Delta(\I)$ and $y\in\Delta(\J)$, $x^{\top}M'y\leq x^{\top}My-c$; take $\min_y$ then $\max_x$. (iii) follows from (ii) applied in both directions with $c=-\|M-M'\|_\infty$. (v): if $x^{\top}M\geq0$ then $x^{\top}My\geq0$ for every $y\in\Delta(\J)$, so $\val(M)\geq\min_yx^{\top}My\geq0$; the second claim is symmetric.
\end{proof}
In particular, \cref{lem:matrixgames}(iii) shows that $M\mapsto\val(M)$ is $1$-Lipschitz, hence continuous.

\subsection{\texorpdfstring{Auxiliary Matrices $W_\lambda(z)$}{Auxiliary Matrices W-lambda(z)}}

For a matrix $M\in\R^{\I\times\J}$ and $x\in\Delta(\I)$, $y\in\Delta(\J)$, set
\[
M(x,y)\defas x^{\top}My=\sum_{i\in\I}\sum_{j\in\J}x_iy_jM_{ij}\in\R ,
\]
and let $q\circ w\in\R^{\I\times\J}$ denote the Hadamard product of $q$ and $w$, i.e., $(q\circ w)_{ij}\defas q_{ij}w_{ij}$. Under the stationary pair $(x,y)$, the scalars $q(x,y)$, $g(x,y)$, and $(q\circ w)(x,y)$ are, respectively, the one-stage absorption probability, the expected non-absorbing stage payoff, and the absorption-weighted expected absorbing payoff.

Following \citet[Definition~1]{OV23} (also see \citealp{AO19}), we associate with an $m\times n$ absorbing game $(g,q,w)$, for $\lambda\in[0,1]$ and $z\in\R$, the matrix $W_\lambda(z)\in\R^{m\times n}$ with entries
\begin{equation}\label{eq:W}
W_\lambda(z)_{ij}\defas \lambda g_{ij}+(1-\lambda)(q\circ w)_{ij}-z\bigl(\lambda+(1-\lambda)q_{ij}\bigr),
\qquad
W_0(z)_{ij}=q_{ij}(w_{ij}-z).
\end{equation}

\begin{lemma}\label{lem:gamma}
For all $x\in\Delta(\I)$, $y\in\Delta(\J)$ and $\lambda\in(0,1]$, the discounted payoff of the stationary pair $(x,y)$ is
\[
\gamma_\lambda(x,y)=\frac{\lambda\,g(x,y)+(1-\lambda)\,(q\circ w)(x,y)}{\lambda+(1-\lambda)\,q(x,y)} .
\]
\end{lemma}

\begin{proof}
Under the stationary pair $(x,y)$ the stages are i.i.d. with law $x \otimes y$. Therefore, splitting off the first stage gives
\[
\gamma_\lambda(x,y)=\lambda\,g(x,y)+(1-\lambda)\bigl[(q\circ w)(x,y)+(1-q(x,y))\,\gamma_\lambda(x,y)\bigr].
\]
The first stage contributes the discounted payoff $\lambda\,g(x,y)$; in the continuation, absorption freezes the play with the expected absorbing payoff $(q\circ w)(x,y)$ while with probability $1-q(x,y)$ the game returns to its initial state, of value $\gamma_\lambda(x,y)$. Solving and using $1-(1-\lambda)(1-q(x,y))=\lambda+(1-\lambda)q(x,y)$ yields the result.
\end{proof}

Combining \cref{eq:W} with \cref{lem:gamma}, a direct expansion gives, for all $x\in\Delta(\I)$, $y\in\Delta(\J)$, $z\in\R$ and $\lambda\in(0,1]$,
\begin{align}
x^{\top}W_\lambda(z)\,y &=\lambda g(x,y)+(1-\lambda)(q\circ w)(x,y)-z\bigl(\lambda+(1-\lambda)q(x,y)\bigr) \nonumber\\
&=\bigl(\gamma_\lambda(x,y)-z\bigr)\bigl(\lambda+(1-\lambda)q(x,y)\bigr),
    \label{eq:bilinear}
\end{align}
where the second equality uses \cref{lem:gamma}.

\begin{lemma}[{\citealp[Theorem~1]{AO19}}]\label{lem:shapleypencil}
For every $\lambda\in(0,1]$, $\val W_\lambda(v_\lambda)=0$. Moreover, the map $z\mapsto\val W_\lambda(z)$ is continuous, strictly decreasing, and satisfies $\val W_\lambda(z)\to\mp\infty$ as $z\to\pm\infty$.
\end{lemma}

\begin{proof}
By \citet[Theorem~2]{Sha53}, there exists an optimal stationary strategy $x^* \in \Delta(\I)$ for Player~1.
By \cref{eq:minimax}, we have $\gamma_\lambda(x^*,y)\geq v_\lambda$ for every $y\in\Delta(\J)$.
Moreover, by \cref{eq:bilinear} with $z=v_\lambda$, this is equivalent to $x^{*\top}W_\lambda(v_\lambda)\,y\geq0$ for every $y\in\Delta(\J)$.
Hence, we have $x^{*\top}W_\lambda(v_\lambda)\geq0$ componentwise. 
By \cref{lem:matrixgames}(v), we have $\val W_\lambda(v_\lambda)\geq0$. Symmetrically, an optimal stationary strategy $y^*$ of Player~2 gives $W_\lambda(v_\lambda)\,y^*\leq0$ componentwise and $\val W_\lambda(v_\lambda)\leq0$. 
Therefore, we have $\val W_\lambda(v_\lambda)=0$.

For $z<z'$, we have
\[
    W_\lambda(z')_{ij}-W_\lambda(z)_{ij}=-(z'-z)\bigl(\lambda+(1-\lambda)q_{ij}\bigr)\in\bigl[-(z'-z),-\lambda(z'-z)\bigr] .
\]
Therefore, by \cref{lem:matrixgames}(ii) and (iii), we get
\begin{equation}\label{eq:strictmono}
\lambda(z'-z)\ \leq\ \val W_\lambda(z)-\val W_\lambda(z')\ \leq\ z'-z .
\end{equation}
Thus, $z\mapsto\val W_\lambda(z)$ is $1$-Lipschitz and strictly decreasing. As $z\to-\infty$, every entry of $W_\lambda(z)$ tends to $+\infty$, and $\val W_\lambda(z)\geq\min_{ij}W_\lambda(z)_{ij}\to+\infty$. Symmetrically $\val W_\lambda(z)\leq\max_{ij}W_\lambda(z)_{ij}\to-\infty$ as $z\to+\infty$.
\end{proof}

\section{Strictly Absorbing Games}\label{sec:uniform}

This section establishes the first ingredient of the proof.
For strictly absorbing games the asymptotic analysis reduces to a single matrix-game equation.
We show that the discounted values converge, as $\lambda\to0^+$, to the unique zero of the map $z\mapsto\val W_0(z)$ (\cref{lem:uniformlimit}).
This yields a simple criterion for the undiscounted value to equal a given algebraic number $\alpha$, namely the existence of strategies that are optimal with value $0$ in the matrix game $W_0(\alpha)$.
\cref{ex:OVtheorem3} illustrates the criterion on a quadratic example.

Throughout this section, $\Gamma=(g,q,w)$ is an $m\times n$ absorbing game with
\[
\delta\defas \min_{i,j}q_{ij}>0
\qquad\text{(strict absorption)},
\]
and we set $R\defas \max_{i,j}\max\bigl(|g_{ij}|,|w_{ij}|\bigr)$.

\begin{lemma}\label{lem:uniformlimit}
The following statements hold:
\begin{enumerate}
\item[\emph{(i)}] The map $\Phi_0:z\mapsto\val W_0(z)$ is continuous and strictly decreasing on $\R$, with $\Phi_0(z)\to\mp\infty$ as $z\to\pm\infty$; in particular, it has a unique zero $z^*$; and
\item[\emph{(ii)}] $v_\lambda\to z^*$ as $\lambda\to0^+$. In particular, the undiscounted value of $\Gamma$ exists and equals $z^*$, the unique solution of $\val\bigl[q_{ij}(w_{ij}-z)\bigr]=0$.
\end{enumerate}
\end{lemma}

\begin{proof}
\emph{Properties of $\Phi_0$.}
For $z<z'$, we have
\[
W_0(z')_{ij}-W_0(z)_{ij}=-q_{ij}(z'-z)\in[-(z'-z),-\delta(z'-z)] .
\]
Therefore, by \cref{lem:matrixgames}(ii) and (iii), we get
\[
\delta(z'-z)\ \leq\ \Phi_0(z)-\Phi_0(z')\ \leq\ z'-z .
\]
Thus, $\Phi_0$ is $1$-Lipschitz and strictly decreasing.
Moreover, using $q_{ij}\geq\delta>0$, we have $\Phi_0(z)\geq\min_{ij}q_{ij}(w_{ij}-z)\to+\infty$ as $z\to-\infty$ and $\Phi_0(z)\leq\max_{ij}q_{ij}(w_{ij}-z)\to-\infty$ as $z\to+\infty$.
Hence, $\Phi_0$ has a unique zero $z^*$, which proves~(i).

\emph{Convergence of $v_\lambda$.}
Along every play, every stage payoff equals some $g_{ij}$ or some $w_{ij}$, and hence lies in $[-R,R]$.
Therefore, every $\lambda$-discounted payoff lies in $[-R,R]$, and in particular $|v_\lambda|\leq R$ for all $\lambda\in(0,1]$.
Moreover, for all $(i,j)$, $z\in\R$ and $\lambda\in(0,1]$, we have
\[
W_\lambda(z)_{ij}-W_0(z)_{ij}=\lambda\bigl(g_{ij}-(q\circ w)_{ij}-z(1-q_{ij})\bigr) ,
\]
so whenever $|z|\leq R$ we have $\|W_\lambda(z)-W_0(z)\|_\infty\leq\lambda\bigl(2R+|z|\bigr)\leq3\lambda R$.
Now let $(\lambda_k)_{k\geq1}\subset(0,1]$ satisfy $\lambda_k\to0$ and $v_{\lambda_k}\to\bar v$ for some $\bar v\in[-R,R]$.
By \cref{lem:shapleypencil}, we have $\val W_{\lambda_k}(v_{\lambda_k})=0$ for every $k$.
Therefore, using that $\Phi_0$ is $1$-Lipschitz, the bound $\|W_{\lambda_k}(v_{\lambda_k})-W_0(v_{\lambda_k})\|_\infty\leq3\lambda_k R$, and \cref{lem:matrixgames}(iii), we get
\[
|\Phi_0(\bar v)|
\leq|\Phi_0(\bar v)-\Phi_0(v_{\lambda_k})|
+\bigl|\val W_0(v_{\lambda_k})-\val W_{\lambda_k}(v_{\lambda_k})\bigr|
\leq|\bar v-v_{\lambda_k}|+3\lambda_kR\ \longrightarrow\ 0 .
\]
Thus, we have $\Phi_0(\bar v)=0$, and hence $\bar v=z^*$.
Since $|v_\lambda|\leq R$ for all $\lambda\in(0,1]$, every sequence $\lambda_k\to0$ has a subsequence along which $v_\lambda$ converges, necessarily to $z^*$.
Therefore, $v_\lambda\to z^*$ as $\lambda\to0^+$, which proves~(ii).
\end{proof}

\begin{example}\label{ex:OVtheorem3}
Fix $k\in\N^*$ with $k\geq2$, and consider the $2\times2$ game of \citet[Theorem~3]{OV23}, in which Player~1 chooses a row and Player~2 a column.
It is given by
\[
q=\begin{pmatrix}\tfrac1k&1\\[2pt] 1&1\end{pmatrix},
\qquad
w=\begin{pmatrix}k&1\\[2pt] 1&k\end{pmatrix},
\qquad
g_{11}=0 .
\]
The entry $(1,1)$ is the only one that does not absorb with probability one; it yields the stage payoff $0$ and absorbs with probability $\tfrac1k$, in which case the absorbing payoff is $k$.
Each of the other three entries absorbs with probability one, so its stage payoff is irrelevant and the play stops at the indicated absorbing payoff.
Since $q_{ij}>0$ for every $(i,j)$, the game is strictly absorbing.
We show that its undiscounted value is $\sqrt k$.
Indeed, the limit auxiliary matrix is
\[
W_0(z)=\bigl(q_{ij}(w_{ij}-z)\bigr)_{ij}
=\begin{pmatrix}\tfrac1k(k-z)&1-z\\[2pt] 1-z&k-z\end{pmatrix}.
\]
Therefore, at $z=\sqrt k$ we have
\[
W_0(\sqrt k)
=\begin{pmatrix}\tfrac1k(k-\sqrt k)&1-\sqrt k\\[2pt] 1-\sqrt k&k-\sqrt k\end{pmatrix}
=(\sqrt k-1)\begin{pmatrix}\tfrac{1}{\sqrt k}&-1\\[2pt] -1&\sqrt k\end{pmatrix} .
\]
The right-hand matrix is singular, since its determinant is $\tfrac{1}{\sqrt k}\cdot\sqrt k-(-1)(-1)=0$.
Let $x^*=y^*=\frac{1}{\sqrt k+1}(\sqrt k,\,1)^{\top}$, i.e., each player plays the first action with probability $\frac{\sqrt k}{\sqrt k+1}$ and the second with probability $\frac{1}{\sqrt k+1}$.
Then we have
\[
x^{*\top}W_0(\sqrt k)=0
\qquad\text{and}\qquad
W_0(\sqrt k)\,y^*=0 .
\]
Therefore, by \cref{lem:matrixgames}(v), we have $\val W_0(\sqrt k)=0$.
Hence $\sqrt k$ is the unique zero of $z\mapsto\val W_0(z)$, and by \cref{lem:uniformlimit} the undiscounted value of the game is $\sqrt k$.
\end{example}

\section{Algebraic Construction}\label{sec:algebra}

This section is the main technical contribution of the paper.
By \cref{sec:uniform}, to realize a given real algebraic number $\alpha$ as an undiscounted value it suffices to construct a strictly absorbing game whose limit auxiliary matrix has value $0$ at $\alpha$, witnessed by a pair of optimal strategies.
We therefore look for rational $m\times m$ matrices $A$ and $Q$, with $Q>0$, such that the matrix $A-\alpha Q$ is singular with strictly positive left and right null vectors $\tilde x$ and $\tilde y$,
\[
\tilde x^{\top}(A-\alpha Q)=0,
\qquad
(A-\alpha Q)\,\tilde y=0,
\qquad
\tilde x,\tilde y>0 .
\]
Thus, normalizing $\tilde x$ and $\tilde y$ to probability vectors yields fully mixed strategies that are optimal with value $0$ in the matrix game $A-\alpha Q$.

We first define the algebraic objects that we use to construct these matrices.

\paragraph{Polynomial.}
We write $\R[z]$ and $\Q[z]$ for the sets of polynomials with real and rational coefficients, respectively.
A polynomial $f=\sum_{k=0}^m a_kz^k$ with $a_m\neq0$ has \emph{degree} $\deg f\defas m$ and \emph{leading coefficient} $a_m$, and is \emph{monic} if $a_m=1$.
We denote by $\PP$ the $m$-dimensional real vector space of polynomials of degree at most $m-1$, and by $\PQ$ its rational subspace.
We say that $\alpha$ is a \emph{root} of $f$ if $f(\alpha)=0$, and we write $\frac{d}{dz}f$ for the derivative of $f$.

\paragraph{Algebraic number.}
A number $\alpha \in \R$ is a real algebraic number of \emph{degree} $m\geq1$ if it is a root of a rational polynomial of degree $m$ and of no nonzero rational polynomial of smaller degree.
Its \emph{minimal polynomial} $\mu \in \Q[z]$ is the unique monic rational polynomial of degree $m$ with $\mu(\alpha)=0$.
Because $\mu$ is the minimal polynomial of $\alpha$, the number $\alpha$ is a simple root of $\mu$, i.e.,
\begin{equation}\label{eq:simple}
\frac{d}{dz} \mu(\alpha)\neq0 .
\end{equation}
Every $f\in\R[z]$ can be written uniquely as $f=s\mu+r$ with $s,r\in\R[z]$ and $\deg r<m$, which is called the \emph{Euclidean division} of $f$ by $\mu$, with \emph{remainder} $r$.
We denote by $\rho\colon\R[z]\to\PP$ the \emph{remainder map}, which sends $f$ to this remainder $r$ such that $f-\rho(f)\in\mu\,\R[z]$ and $\deg\rho(f)\leq m-1$ for every $f\in\R[z]$.

\subsection{Remainder Map and Lagrange Interpolation Polynomial}

In this subsection, we first present the algebraic properties of the remainder map $\rho$ (\cref{lem:rho}). We then define the Lagrange interpolation polynomial $\ell$ determined by the polynomial $\mu$ and normalized such that $\ell(\alpha)=1$.
Its key property is the absorption identity $\rho(f\ell)=f(\alpha)\ell$ (\cref{lem:absorption}), which reduces multiplication by $\ell$ followed by the remainder map to evaluation at $\alpha$.

\begin{lemma}\label{lem:rho}
The remainder map $\rho$ is linear over $\R$. Moreover, it satisfies, for all $f,h\in\R[z]$ and $u\in\PP$:
\begin{enumerate}
\item[\emph{(i)}] $\rho(u)=u$;
\item[\emph{(ii)}] $\rho(f)(\alpha)=f(\alpha)$;
\item[\emph{(iii)}] $\rho(f\,\rho(hu))=\rho(fhu)$; and
\item[\emph{(iv)}] $\rho(\Q[z])\subseteq\PQ$.
\end{enumerate}
\end{lemma}

\begin{proof}
Linearity and (i) are clear from uniqueness of Euclidean division. (ii): $f-\rho(f)\in\mu\R[z]$ vanishes at $\alpha$ because $\mu(\alpha)=0$. (iii): $f\rho(hu)-fhu=f\cdot(\rho(hu)-hu)\in\mu\R[z]$, and two polynomials differing by a multiple of $\mu$ have the same remainder. (iv): Euclidean division of a rational polynomial by the rational polynomial $\mu$ has rational quotient and remainder.
\end{proof}

Since $\alpha$ is a root of $\mu$, we write $\mu(z)=(z-\alpha)\mu_1(z)$ with $\mu_1\in\R[z]$ of degree $m-1$. Differentiating gives $\frac{d}{dz} \mu(z)=\mu_1(z)+(z-\alpha)\frac{d}{dz} \mu_1(z)$. Therefore, we have $\mu_1(\alpha)=\frac{d}{dz} \mu(\alpha) \neq0$ by \cref{eq:simple}.

\paragraph{Lagrange interpolation polynomial.} We define the \emph{Lagrange interpolation polynomial}
\begin{equation}\label{eq:lagrange}
\ell(z)\defas \frac{\mu(z)}{(z-\alpha)\,\mu_1(\alpha)}=\frac{\mu_1(z)}{\mu_1(\alpha)}\ \in\ \PP .
\end{equation}
The second form shows $\ell(\alpha)=1$, and multiplying by $z-\alpha$ gives $(z-\alpha)\,\ell(z)=\mu(z)/\mu_1(\alpha)$.

\begin{lemma}\label{lem:absorption}
For every $f\in\R[z]$,
\[
\rho\bigl(f\,\ell\bigr)=f(\alpha)\,\ell .
\]
\end{lemma}

\begin{proof}
We have $f(z)-f(\alpha)=(z-\alpha)h(z)$ for some $h\in\R[z]$. Hence, by \cref{eq:lagrange},
\[
f(z)\ell(z)-f(\alpha)\ell(z)=(z-\alpha)\ell(z)\,h(z)=\frac{\mu(z)h(z)}{\mu_1(\alpha)}\in\mu\,\R[z],
\]
Therefore, $\rho(f\ell)=\rho\bigl(f(\alpha)\ell\bigr)=f(\alpha)\ell$, where the last equality holds because $f(\alpha)\ell\in\PP$ (\cref{lem:rho}(i)).
\end{proof}

\subsection{Positive Rational Basis}

Our construction requires a rational basis $(p_1,\dots,p_m) \in \big(\PQ \big)^m$ of $\PP$ in which both the values $p_k(\alpha)$ of the basis polynomials are strictly positive. Moreover, the coordinates of the Lagrange interpolation polynomial~$\ell$ in this basis are strictly positive.
These two families of numbers become the strictly positive null vectors $\tilde x$ and $\tilde y$ sought above (the fully mixed optimal strategies).
A generic rational basis achieves neither positivity property.
The following lemma produces one that achieves both, by perturbing $\ell$ and using density of the rational bases.

\begin{lemma}\label{lem:basis}
There exist $p_1,\dots,p_m\in\PQ$ and real numbers $\tilde y_1,\dots,\tilde y_m$ such that:
\begin{enumerate}
\item[\emph{(i)}] $(p_1,\dots,p_m)$ is a basis of $\PP$ over $\R$ (hence of $\PQ$ over $\Q$);
\item[\emph{(ii)}] $p_k(\alpha)>0$ for every $k$;
\item[\emph{(iii)}] $\ell=\sum_{k=1}^m\tilde y_kp_k$ with $\tilde y_k>0$ for every $k$.
\end{enumerate}
\end{lemma}

\begin{proof}
Consider the set
\begin{align*}
\mathcal U\defas \Bigl\{(f_1,\dots,f_m)\in(\PP)^m:\ &(f_k)_k\text{ is a basis of }\PP\\
&\land f_k(\alpha)>0\ \forall k\\
&\land \text{the coordinates of }\ell\text{ in }(f_k)_k\text{ are all }>0\Bigr\}.
\end{align*}

We first show that $\mathcal U$ is open in $(\PP)^m$.
Indeed, let $(f_1,\dots,f_m)$ be any tuple in $(\PP)^m$ and let $F\in\R^{m\times m}$ be the matrix whose $k$-th column is the coefficient vector of $f_k$.
Note that $\mathcal U$ is constrained by the three conditions in its definition.
We check that each is an open condition on $F$, so that $\mathcal U$, their intersection, is open.
First, $(f_1,\dots,f_m)$ is a basis if and only if $\det F\neq0$, which is open because $\det$ is continuous.
Second, each value $f_k(\alpha)$ is a linear, hence continuous, function of $F$, so the requirement $f_k(\alpha)>0$ for all $k$ is open.
Third, on the open set where $F$ is invertible, the coordinates $c_1,\dots,c_m$ of the fixed polynomial $\ell$, defined by $\ell=\sum_kc_kf_k$, are given by Cramer's rule as continuous functions of $F$, so the requirement $c_k>0$ for all $k$ is open.
Therefore, $\mathcal U$ is open in $(\PP)^m$.

We now show that $\mathcal U$ is nonempty. Let
\[
\mathcal Z\defas \{f\in\PP:f(\alpha)=0\}.
\]
This is the kernel of the nonzero linear map $f\mapsto f(\alpha)$. Hence, we have
$\dim\mathcal Z=m-1$. Also $\ell\notin\mathcal Z$, since $\ell(\alpha)=1$.
Choose a basis $h_1,\dots,h_{m-1}$ of $\mathcal Z$ (an empty list if $m=1$), and set
\[
u_k\defas h_k\ (1\leq k\leq m-1),
\qquad
u_m\defas -(h_1+\dots+h_{m-1}),
\]
with the convention that the empty sum is $0$. Then $u_1,\dots,u_m\in\mathcal Z$ and
\[
u_1+\dots+u_m=0.
\]
Define $f_k\defas \ell+u_k$ for $1\leq k\leq m$. Since $u_k(\alpha) = 0$, we have
\[
f_k(\alpha)=\ell(\alpha)+u_k(\alpha)=1>0
\qquad(1\leq k\leq m).
\]
Moreover, we have
\begin{align*}
    \frac1m\sum_{k=1}^mf_k &= \frac1m\sum_{k=1}^m(\ell+u_k)\\
    &= \frac1m\sum_{k=1}^m\ell + \frac1m\sum_{k=1}^mu_k\\
    &= \ell .
\end{align*}
Thus, if $(f_1,\dots,f_m)$ is a basis, then the coordinates of $\ell$ in this basis are all equal to $\frac1m>0$.

It remains to prove that $(f_1,\dots,f_m)$ is a basis. Suppose
$\sum_{k=1}^mc_kf_k=0$. Evaluating at $\alpha$ gives $\sum_kc_k=0$. Hence
\[
0=\sum_{k=1}^mc_kf_k
=\Bigl(\sum_{k=1}^mc_k\Bigr)\ell+\sum_{k=1}^mc_ku_k
=\sum_{k=1}^mc_ku_k
=\sum_{k=1}^{m-1}(c_k-c_m)h_k .
\]
Since $h_1,\dots,h_{m-1}$ are linearly independent, $c_k=c_m$ for all
$k\leq m-1$. Together with $\sum_kc_k=0$, this gives $c_1=\dots=c_m=0$.
Therefore $(f_1,\dots,f_m)$ is a basis.
Hence $(f_1,\dots,f_m)\in\mathcal U$.

Since $(\PQ)^m$ is dense in $(\PP)^m$, the nonempty open set $\mathcal U$ contains a rational tuple $(p_1,\dots,p_m)\in(\PQ)^m$. Properties (ii) and (iii) hold by definition of $\mathcal U$. The tuple $(p_1,\dots,p_m)$ is a basis of $\PP$ over $\R$ by definition of $\mathcal U$; since the $p_k$ are rational polynomials and are linearly independent over $\R$, they are linearly independent over $\Q$, and as $\dim_\Q\PQ=m$, they form a basis of $\PQ$ over $\Q$. Thus property (i) holds, which completes the proof.
\end{proof}

\paragraph{Positive rational basis.}
From now on we fix a basis $(p_1,\dots,p_m)$ as in \cref{lem:basis} and set
\begin{equation}\label{eq:pie}
\tilde x\defas \bigl(p_1(\alpha),\dots,p_m(\alpha)\bigr)^{\top}>0,
\qquad
\tilde y\defas (\tilde y_1,\dots,\tilde y_m)^{\top}>0 .
\end{equation}
Note that $\ell=\sum_{k=1}^m\tilde y_kp_k$.

\subsection{Multiplication Matrices}

With the positive rational basis $(p_1,\dots,p_m)$ fixed, we encode multiplication by a polynomial followed by taking the remainder modulo $\mu$ in this basis.
Specifically, to each $f\in\R[z]$ we associate the matrix $M(f)$ of the multiplication-by-$f$ map on $\PP$ written in the basis.
We then present the basic properties of these matrices for every polynomial $f$ (\cref{lem:mult}).
We finally show that there exists a rational polynomial $\xi \in \PQ$ whose multiplication matrix satisfies $M(\xi)>0$, and when $\alpha>0$, $\xi$ is chosen such that also $M\bigl(z\,\xi(z)\bigr)>0$ (\cref{prop:xi}).
The matrices $A$ and $Q$ sought above are obtained from this polynomial $\xi$.

\paragraph{Multiplication matrices.}
For $f\in\R[z]$, define $M(f)\in\R^{m\times m}$ as the unique matrix satisfying
\begin{equation}\label{eq:Mdef}
\rho(f\,p_j)=\sum_{k=1}^mM(f)_{kj}\,p_k
\qquad(1\leq j\leq m).
\end{equation}
The uniqueness follows because $(p_1,\dots,p_m)$ is a basis of $\PP$.

\begin{lemma}\label{lem:mult}
The map $f\mapsto M(f)$ has the following properties. For all $f,h\in\R[z]$:
\begin{enumerate}
\item[\emph{(i)}] it is linear over $\R$, and $M(1)$ is the identity matrix;
\item[\emph{(ii)}] if $f\in\Q[z]$, then $M(f)\in\Q^{m\times m}$;
\item[\emph{(iii)}] the vector $\tilde x$ satisfies
\[
\tilde x^{\top}M(f)=f(\alpha)\,\tilde x^{\top};
\]
\item[\emph{(iv)}] the vector $\tilde y$ satisfies
\[
M(f)\,\tilde y=f(\alpha)\,\tilde y;
\]
\item[\emph{(v)}] for the interpolation polynomial $\ell$, we have
\[
M(\ell)=\tilde y\,\tilde x^{\top};
\]
in particular $M(\ell)>0$; and 
\item[\emph{(vi)}] $M(fh)=M(f)\,M(h)$.
\end{enumerate}
\end{lemma}

\begin{proof}
(i) Let $a,b\in\R$. By the linearity of $\rho$, for every $j$ we have
\[
\rho\bigl((af+bh)p_j\bigr)=a\,\rho(fp_j)+b\,\rho(hp_j).
\]
Comparing the coefficients of $p_1,\dots,p_m$ gives
$M(af+bh)=aM(f)+bM(h)$. Also, $\rho(p_j)=p_j$ by \cref{lem:rho}(i), so the $j$-th column of $M(1)$ is the $j$-th unit vector. Thus $M(1)$ is the identity matrix.

(ii) If $f\in\Q[z]$, then $fp_j\in\Q[z]$, and hence $\rho(fp_j)\in\PQ$ by \cref{lem:rho}(iv). Since $(p_1,\dots,p_m)$ is a basis of $\PQ$ over $\Q$, the coefficients of $\rho(fp_j)$ in this basis are rational. These coefficients are precisely the entries of the $j$-th column of $M(f)$.

(iii) Evaluate \cref{eq:Mdef} at $\alpha$. By \cref{lem:rho}(ii), the left-hand side is
\[
\rho(fp_j)(\alpha)=f(\alpha)p_j(\alpha)=f(\alpha)\tilde x_j .
\]
The right-hand side is
\[
\sum_kM(f)_{kj}p_k(\alpha)=\sum_kM(f)_{kj}\tilde x_k
=\bigl(\tilde x^{\top}M(f)\bigr)_j .
\]
Since this holds for every $j$, we obtain $\tilde x^{\top}M(f)=f(\alpha)\tilde x^{\top}$.

(iv) Since $\ell=\sum_{j=1}^m\tilde y_jp_j$, the linearity of $\rho$ and \cref{eq:Mdef} give
\[
\rho(f\ell)
=\sum_{j=1}^m\tilde y_j\rho(fp_j)
=\sum_{j=1}^m\tilde y_j\sum_{k=1}^mM(f)_{kj}p_k
=\sum_{k=1}^m\bigl(M(f)\tilde y\bigr)_kp_k .
\]
By \cref{lem:absorption},
\[
\rho(f\ell)=f(\alpha)\ell=\sum_{k=1}^mf(\alpha)\tilde y_kp_k .
\]
Comparing the coefficients of $p_1,\dots,p_m$ gives $M(f)\tilde y=f(\alpha)\tilde y$.

(v) Fix $j\in[m]$. By definition, the $j$-th column of $M(\ell)$ is given by the coefficients of $\rho(\ell p_j)$ in the basis $(p_k)_k$. Applying \cref{lem:absorption} to $p_j$ gives
\[
\rho(\ell p_j)=\rho(p_j\ell)=p_j(\alpha)\ell
=\tilde x_j\sum_{k=1}^m\tilde y_kp_k .
\]
Thus the $j$-th column of $M(\ell)$ is $\tilde x_j\tilde y$. Since this holds for every $j$, we have $M(\ell)=\tilde y\tilde x^{\top}$. By \cref{eq:pie}, this gives $M(\ell)>0$.

(vi) For each $j$, \cref{lem:rho}(iii) and \cref{eq:Mdef} give
\begin{align*}
\rho(fh\,p_j)
&=\rho\bigl(f\,\rho(hp_j)\bigr)\\
&=\sum_r M(h)_{rj}\rho(fp_r)\\
&=\sum_r M(h)_{rj}\sum_kM(f)_{kr}p_k\\
&=\sum_k\bigl(M(f)M(h)\bigr)_{kj}p_k .
\end{align*}
Comparing this with the definition of $M(fh)$ gives $M(fh)=M(f)M(h)$.
\end{proof}

\begin{corollary}\label{prop:xi}
There exists $\xi\in\PQ$ such that $M(\xi)>0$. If $\alpha>0$, then $\xi$ can moreover be chosen so that $M\bigl(z\,\xi(z)\bigr)>0$.
\end{corollary}

\begin{proof}
For each pair $(k,j)$, the entry $M(f)_{kj}$ depends linearly on $f\in\PP$.
Indeed, in the definition
\[
\rho(fp_j)=\sum_{r=1}^mM(f)_{rj}p_r ,
\]
multiplication by $p_j$, the remainder map $\rho$, and taking the coefficient of $p_k$ in the basis $(p_1,\dots,p_m)$ are all linear operations.
Hence each map $f\mapsto M(f)_{kj}$ is continuous.

Now set
\[
\mathcal V\defas \{f\in\PP:M(f)>0\}.
\]
Since there are only finitely many entries and each entry depends continuously on $f$, the set $\mathcal V$ is open in $\PP$.

By \cref{lem:mult}(v), we have $M(\ell)=\tilde y\tilde x^{\top}$.
Since $\tilde x,\tilde y>0$ by \cref{eq:pie}, we have $M(\ell)>0$. Hence, we have $\ell\in\mathcal V$.
Since $\PQ$ is dense in $\PP$, the nonempty open set $\mathcal V$ contains some $\xi\in\PQ$.
For this choice of $\xi$, we have $M(\xi)>0$.

Assume now that $\alpha>0$. We claim that the rational polynomial $\xi$ can be chosen so that also $M\bigl(z\,\xi(z)\bigr)>0$.
Set
\[
\mathcal V_+\defas \{f\in\PP:M(f)>0\ \text{and}\ M\bigl(z\,f(z)\bigr)>0\}.
\]
As above, the entries of both $M(f)$ and $M\bigl(z\,f(z)\bigr)$ depend continuously on $f$, so $\mathcal V_+$ is open in $\PP$.
Moreover, $\ell\in\mathcal V_+$: we already know that $M(\ell)>0$, and by \cref{lem:mult}(vi), (iv), and (v),
\[
M(z\ell)=M(z)M(\ell)
=\bigl(M(z)\tilde y\bigr)\tilde x^{\top}
=\alpha\,\tilde y\tilde x^{\top}>0 .
\]
Therefore, by density of $\PQ$ in $\PP$, the nonempty open set $\mathcal V_+$ contains some $\xi\in\PQ$.
For this choice of $\xi$, both $M(\xi)>0$ and $M\bigl(z\,\xi(z)\bigr)>0$, as claimed.
\end{proof}

\subsection{\texorpdfstring{Matrices $A$ and $Q$ and Their Positive Null Vectors}{Matrices A and Q and Their Positive Null Vectors}}

With the rational polynomial $\xi\in\PQ$ of \cref{prop:xi} fixed, we form the two matrices sought at the start of the section.
We take $Q$ and $A$ to be the multiplication matrices of $\xi$ and of $z\,\xi(z)$, both rational and with $Q>0$. Moreover, when $\alpha>0$, the choice of $\xi$ also gives $A>0$.
Evaluating the identities $\tilde x^{\top}M(f)=f(\alpha)\tilde x^{\top}$ and $M(f)\tilde y=f(\alpha)\tilde y$ at $f=\xi$ and $f=z\,\xi(z)$ gives $\tilde x^{\top}(A-\alpha Q)=0$ and $(A-\alpha Q)\tilde y=0$ (\cref{thm:nullvectors}).
Thus, $\tilde x$ and $\tilde y$ are strictly positive left and right null vectors of $A-\alpha Q$, as required.

\paragraph{Matrices $A$ and $Q$.} Fix the rational polynomial $\xi\in\PQ$ as in \cref{prop:xi} and define
\begin{equation}\label{eq:AQ}
Q\defas M(\xi)\in\Q^{m\times m},\qquad
A\defas M\bigl(z\,\xi(z)\bigr)\in\Q^{m\times m} .
\end{equation}
Indeed $z\,\xi(z)\in\Q[z]$, so $A$ is rational by \cref{lem:mult}(ii), and $Q$ is rational with $Q>0$. If $\alpha>0$, then $A>0$ by the choice of $\xi$ in \cref{prop:xi}.

\begin{lemma}\label{thm:nullvectors}
With $\tilde x,\tilde y>0$ as in \cref{eq:pie} and $A,Q$ as in \cref{eq:AQ},
\[
\tilde x^{\top}(A-\alpha Q)=0
\qquad\text{and}\qquad
(A-\alpha Q)\,\tilde y=0 .
\]
\end{lemma}

\begin{proof}
By \cref{lem:mult}(iii) applied to $f(z)=z\,\xi(z)$ and to $f=\xi$, we have
\[
\tilde x^{\top}A=\alpha\,\xi(\alpha)\,\tilde x^{\top},
\qquad
\tilde x^{\top}Q=\xi(\alpha)\,\tilde x^{\top}.
\]
Therefore, $\tilde x^{\top}(A-\alpha Q)=\bigl(\alpha\xi(\alpha)-\alpha\,\xi(\alpha)\bigr)\tilde x^{\top}=0$. Likewise, by \cref{lem:mult}(iv),
$A\,\tilde y=\alpha\,\xi(\alpha)\,\tilde y$ and $Q\,\tilde y=\xi(\alpha)\,\tilde y$. Thus, $(A-\alpha Q)\tilde y=0$.
\end{proof}

\section{Proof of Main Result}\label{sec:proof}

\begin{theorem}[Main theorem]\label{thm:main}
Let $m\geq1$ and let $\alpha\in\R$ be an algebraic number of degree $m$ over $\Q$. Let $\tilde x,\tilde y,A,Q$ be as in \cref{sec:algebra}, let $N\in\N^*$ satisfy $N\geq\max_{i,j}Q_{ij}$, and define the $m\times m$ absorbing game $\Gamma_\alpha=(g,q,w)$ by
\[
g_{ij}\defas 0,
\qquad
q_{ij}\defas \frac{Q_{ij}}{N}\in(0,1]\cap\Q,
\qquad
w_{ij}\defas \frac{A_{ij}}{Q_{ij}}\in\Q
\qquad
\text{for all }(i,j)\in \I\times\J .
\]
Then $\Gamma_\alpha$ is a rational, strictly absorbing, and recursive $m\times m$ absorbing game, and its undiscounted value equals $\alpha$. If $\alpha>0$, then $\Gamma_\alpha$ is moreover positive recursive. In particular, \cref{conj:OV} holds.
\end{theorem}

\begin{proof}
The game is rational because $A$ and $Q$ are rational. It is strictly absorbing because, for every $(i,j)$,
\[
q_{ij}=\frac{Q_{ij}}{N}>0 ,
\]
and $q_{ij}\leq1$ by the choice of $N$. It is recursive because $g_{ij}=0$ for every $(i,j)$.
If $\alpha>0$, then $A>0$ by the choice of $\xi$ in \cref{prop:xi}. Since also $Q>0$, we have
\[
w_{ij}=\frac{A_{ij}}{Q_{ij}}>0
\qquad\text{for every }(i,j),
\]
so the recursive game $\Gamma_\alpha$ is positive recursive.

For $z\in\R$, the limit auxiliary matrix defined in \cref{eq:W} is
\[
W_0(z)_{ij}=q_{ij}(w_{ij}-z)=\frac1N\bigl(A_{ij}-zQ_{ij}\bigr),
\]
and therefore
\[
W_0(z)=\frac1N(A-zQ).
\]
Since $\tilde x,\tilde y>0$, define
\[
x^*\defas \frac{\tilde x}{\sum_k\tilde x_k}\in\Delta(\I),
\qquad
y^*\defas \frac{\tilde y}{\sum_k\tilde y_k}\in\Delta(\J).
\]
By \cref{thm:nullvectors},
\[
x^{*\top}W_0(\alpha)=\frac{1}{N\sum_k\tilde x_k}\,\tilde x^{\top}(A-\alpha Q)=0,
\qquad
W_0(\alpha)\,y^*=\frac{1}{N\sum_k\tilde y_k}\,(A-\alpha Q)\,\tilde y=0 .
\]
Thus \cref{lem:matrixgames}(v) gives $\val W_0(\alpha)=0$.
Since the game is strictly absorbing, \cref{lem:uniformlimit} implies that $\alpha$ is the undiscounted value of $\Gamma_\alpha$.

The game has $m$ actions for each player and $\deg\alpha=m$, so this proves \cref{conj:OV} for the given real algebraic number $\alpha$, which completes the proof.
\end{proof}

\cref{thm:main}, action duplication, and the degree bound of \citet[Proposition 1]{OB21} imply the following result.

\begin{corollary}\label{cor:char}
For every $m\geq1$,
\begin{multline*}
\bigl\{\text{undiscounted values of rational $m\times m$ absorbing games}\bigr\}\\
=\bigl\{\alpha\in\R:\alpha\text{ algebraic of degree at most }m\bigr\}.
\end{multline*}
\end{corollary}

\begin{proof}
The inclusion from left to right is \citet[Proposition 1]{OB21}. Conversely, let $\alpha\in\R$ be algebraic of degree $k\leq m$. By \cref{thm:main}, $\alpha$ is the undiscounted value of a rational $k\times k$ absorbing game. If $k<m$, duplicate arbitrary actions of each player until both players have $m$ actions. This does not change any discounted value: applying, at every history, the projection that sums the probabilities assigned to duplicate actions maps strategies in the enlarged game to strategies in the original game with the same induced payoff and transition law, and strategies in the original game lift to the enlarged game by assigning duplicate actions probability zero. Hence the undiscounted value is unchanged, so $\alpha$ is the undiscounted value of a rational $m\times m$ absorbing game.
\end{proof}

\section{Explicit Examples}\label{sec:example}

We illustrate the construction on two cubic examples. For $\alpha=\sqrt[3]3$ the data produced are highly symmetric; the second example, the root of $z^3-z-1$, is less symmetric.

\subsection{The Cube Root of Three}

Let $\alpha=\sqrt[3]3\in(1,2)$ be the real root of
\[
\mu(z)=z^3-3 .
\]
The polynomial $\mu$ has no rational root, so it is the minimal polynomial of $\alpha$ and $m=3$.

\paragraph{Positive rational basis.}
The power basis already has the required positivity. Take
\[
p_1(z)=1,\qquad p_2(z)=z,\qquad p_3(z)=z^2,
\]
so that
\[
\tilde x=\begin{pmatrix}p_1(\alpha)\\ p_2(\alpha)\\ p_3(\alpha)\end{pmatrix}
=\begin{pmatrix}1\\ \alpha\\ \alpha^2\end{pmatrix}>0 .
\]
Since $\mu(z)=(z-\alpha)(z^2+\alpha z+\alpha^2)$ and $\frac{d}{dz}\mu(\alpha)=3\alpha^2$, the interpolation polynomial is
\[
\ell(z)=\frac{z^2+\alpha z+\alpha^2}{3\alpha^2},
\]
whose coordinates in the basis $(1,z,z^2)$ are positive for every $\alpha>0$:
\[
\ell=\tilde y_1p_1+\tilde y_2p_2+\tilde y_3p_3,
\qquad
\tilde y=\frac1{3\alpha^2}
\begin{pmatrix}
\alpha^2\\[2pt]
\alpha\\[2pt]
1
\end{pmatrix}>0 .
\]

\paragraph{Multiplication matrices.}
Modulo $\mu$, we have
\[
z^3\equiv 3,
\qquad
z^4\equiv 3z .
\]
Hence, if $f(z)=a_0+a_1z+a_2z^2$, then
\[
M(f)=
\begin{pmatrix}
a_0 & 3a_2 & 3a_1\\
a_1 & a_0 & 3a_2\\
a_2 & a_1 & a_0
\end{pmatrix}.
\]
Choose
\[
\xi(z)=1+z+z^2\in\PQ .
\]
Then
\[
Q=M(\xi)=
\begin{pmatrix}
1&3&3\\
1&1&3\\
1&1&1
\end{pmatrix}>0 .
\]
Also,
\[
z\,\xi(z)=z+z^2+z^3\equiv 3+z+z^2\pmod{\mu},
\]
and therefore
\[
A=M\bigl(z\,\xi(z)\bigr)=
\begin{pmatrix}
3&3&3\\
1&3&3\\
1&1&3
\end{pmatrix}>0 .
\]

\paragraph{The absorbing game.}
Since all entries of $Q$ are at most $3$, we may take $N=4$ in \cref{thm:main}. Thus
\[
g_{ij}=0,
\qquad
q_{ij}=\frac{Q_{ij}}4,
\qquad
w_{ij}=\frac{A_{ij}}{Q_{ij}} ,
\]
that is,
\[
q=
\begin{pmatrix}
\frac14&\frac34&\frac34\\[2pt]
\frac14&\frac14&\frac34\\[2pt]
\frac14&\frac14&\frac14
\end{pmatrix},
\qquad
w=
\begin{pmatrix}
3&1&1\\[2pt]
1&3&1\\[2pt]
1&1&3
\end{pmatrix}.
\]
In the notation of \citet{OV23}, where an entry $\bigl(g;(q,w^*)\bigr)$ lists the stage payoff, the absorption probability and the absorbing payoff, the resulting game is
\[
\Gamma_{\alpha}\ =\
\begin{pmatrix}
\bigl(0;(\tfrac14,3^*)\bigr) & \bigl(0;(\tfrac34,1^*)\bigr) & \bigl(0;(\tfrac34,1^*)\bigr)\\[6pt]
\bigl(0;(\tfrac14,1^*)\bigr) & \bigl(0;(\tfrac14,3^*)\bigr) & \bigl(0;(\tfrac34,1^*)\bigr)\\[6pt]
\bigl(0;(\tfrac14,1^*)\bigr) & \bigl(0;(\tfrac14,1^*)\bigr) & \bigl(0;(\tfrac14,3^*)\bigr)
\end{pmatrix}.
\]
This game is positive recursive: all non-absorbing stage payoffs are zero, and all absorbing payoffs are positive.

\paragraph{Verification of the value.}
For this game,
\[
W_0(z)=\frac14(A-zQ),
\qquad
A-\alpha Q=
\begin{pmatrix}
3-\alpha&3-3\alpha&3-3\alpha\\
1-\alpha&3-\alpha&3-3\alpha\\
1-\alpha&1-\alpha&3-\alpha
\end{pmatrix}.
\]
Using $\alpha^3=3$, the left identity is
\[
\tilde x^{\top}(A-\alpha Q)=\bigl(3-\alpha^3\bigr)\begin{pmatrix}1&1&1\end{pmatrix}=0 .
\]
For the right identity, with $\hat y\defas(\alpha^2,\alpha,1)^{\top}$ and $\tilde y=(3\alpha^2)^{-1}\hat y$,
\[
(A-\alpha Q)\hat y=
\begin{pmatrix}
(3-\alpha)\alpha^2+(3-3\alpha)\alpha+(3-3\alpha)\\
(1-\alpha)\alpha^2+(3-\alpha)\alpha+(3-3\alpha)\\
(1-\alpha)\alpha^2+(1-\alpha)\alpha+(3-\alpha)
\end{pmatrix}
=\bigl(3-\alpha^3\bigr)\begin{pmatrix}1\\1\\1\end{pmatrix}=0 .
\]
Normalizing $\tilde x$ and $\tilde y$ to probability vectors
\[
x^*\defas \frac{\tilde x}{1+\alpha+\alpha^2},
\qquad
y^*\defas \frac{\tilde y}{\tilde y_1+\tilde y_2+\tilde y_3},
\]
gives $x^{*\top}(A-\alpha Q)=0$ and $(A-\alpha Q)y^*=0$. Hence, by \cref{lem:matrixgames}(v), $\val(A-\alpha Q)=0$, and therefore $\val W_0(\alpha)=0$. Since the game is strictly absorbing, \cref{lem:uniformlimit} implies that the undiscounted value of $\Gamma_\alpha$ is $\alpha=\sqrt[3]3$. Finally, the determinant
\[
\det(A-tQ)=-4\,(t^3-3)
\]
recovers the minimal polynomial $t^3-3$.

\subsection{\texorpdfstring{A Less Symmetric Cubic: $\alpha^3-\alpha-1=0$}{A Less Symmetric Cubic: alpha cubed minus alpha minus 1 equals 0}}

We now turn to the root $\alpha\in(1,2)$ of
\[
\mu(z)=z^3-z-1 .
\]
Such a root exists because $\mu(1)<0$ and $\mu(2)>0$. The polynomial $\mu$ has no rational root, so it is the minimal polynomial of $\alpha$ and $m=3$.

\paragraph{Positive rational basis.}
For this example, the usual power basis already has the required positivity. Take
\[
p_1(z)=1,\qquad p_2(z)=z,\qquad p_3(z)=z^2 .
\]
Then
\[
\tilde x=\begin{pmatrix}p_1(\alpha)\\ p_2(\alpha)\\ p_3(\alpha)\end{pmatrix}
=\begin{pmatrix}1\\ \alpha\\ \alpha^2\end{pmatrix}>0 .
\]
Moreover,
\[
\mu(z)=(z-\alpha)(z^2+\alpha z+\alpha^2-1),
\qquad
\frac{d}{dz}\mu(\alpha)=3\alpha^2-1 .
\]
Thus the interpolation polynomial is
\[
\ell(z)=\frac{z^2+\alpha z+\alpha^2-1}{3\alpha^2-1}.
\]
Since $\alpha>1$, all coefficients of $\ell$ in the basis $(1,z,z^2)$ are positive:
\[
\ell=\tilde y_1p_1+\tilde y_2p_2+\tilde y_3p_3,
\qquad
\tilde y=\frac1{3\alpha^2-1}
\begin{pmatrix}
\alpha^2-1\\[2pt]
\alpha\\[2pt]
1
\end{pmatrix}>0 .
\]

\paragraph{Multiplication matrices.}
Modulo $\mu$, we have
\[
z^3\equiv z+1,
\qquad
z^4\equiv z^2+z .
\]
Hence, if $f(z)=a_0+a_1z+a_2z^2$, then
\[
M(f)=
\begin{pmatrix}
a_0 & a_2 & a_1\\
a_1 & a_0+a_2 & a_1+a_2\\
a_2 & a_1 & a_0+a_2
\end{pmatrix}.
\]
Indeed, the three columns are the coordinates of $\rho(f)$, $\rho(zf)$ and $\rho(z^2f)$ in the basis $(1,z,z^2)$.

Choose
\[
\xi(z)=1+z+z^2\in\PQ .
\]
Then
\[
Q=M(\xi)=
\begin{pmatrix}
1&1&1\\
1&2&2\\
1&1&2
\end{pmatrix}>0 .
\]
Also,
\[
z\,\xi(z)=z+z^2+z^3\equiv 1+2z+z^2\pmod{\mu},
\]
and therefore
\[
A=M\bigl(z\,\xi(z)\bigr)=
\begin{pmatrix}
1&1&2\\
2&2&3\\
1&2&2
\end{pmatrix}>0 .
\]

\paragraph{The absorbing game.}
Since all entries of $Q$ are at most $2$, we may take $N=3$ in \cref{thm:main}. Thus
\[
g_{ij}=0,
\qquad
q_{ij}=\frac{Q_{ij}}3,
\qquad
w_{ij}=\frac{A_{ij}}{Q_{ij}} .
\]
Explicitly,
\[
q=
\begin{pmatrix}
\frac13&\frac13&\frac13\\[2pt]
\frac13&\frac23&\frac23\\[2pt]
\frac13&\frac13&\frac23
\end{pmatrix},
\qquad
w=
\begin{pmatrix}
1&1&2\\[2pt]
2&1&\frac32\\[2pt]
1&2&1
\end{pmatrix}.
\]
In the same notation, the resulting game is
\[
\Gamma_{\alpha}\ =\
\begin{pmatrix}
\bigl(0;(\tfrac13,1^*)\bigr) & \bigl(0;(\tfrac13,1^*)\bigr) & \bigl(0;(\tfrac13,2^*)\bigr)\\[6pt]
\bigl(0;(\tfrac13,2^*)\bigr) & \bigl(0;(\tfrac23,1^*)\bigr) & \bigl(0;(\tfrac23,(\tfrac32)^*)\bigr)\\[6pt]
\bigl(0;(\tfrac13,1^*)\bigr) & \bigl(0;(\tfrac13,2^*)\bigr) & \bigl(0;(\tfrac23,1^*)\bigr)
\end{pmatrix}.
\]
This game is positive recursive: all non-absorbing stage payoffs are zero, and all absorbing payoffs are positive.

\paragraph{Verification of the value.}
For this game,
\[
W_0(z)=\frac13(A-zQ).
\]
We verify directly that $\alpha$ is the zero selected by the construction. First,
\[
A-\alpha Q=
\begin{pmatrix}
1-\alpha&1-\alpha&2-\alpha\\
2-\alpha&2-2\alpha&3-2\alpha\\
1-\alpha&2-\alpha&2-2\alpha
\end{pmatrix}.
\]
Using $\alpha^3=\alpha+1$, we get
\[
\tilde x^{\top}(A-\alpha Q)
=\begin{pmatrix}
1+\alpha-\alpha^3&
1+\alpha-\alpha^3&
2+2\alpha-2\alpha^3
\end{pmatrix}
=0 .
\]
For the right identity, it is enough to use the vector
\[
\hat y\defas
\begin{pmatrix}
\alpha^2-1\\ \alpha\\ 1
\end{pmatrix},
\]
since $\tilde y=(3\alpha^2-1)^{-1}\hat y$. Again using $\alpha^3=\alpha+1$, we have
\[
(A-\alpha Q)\hat y=
\begin{pmatrix}
(1-\alpha)(\alpha^2-1)+(1-\alpha)\alpha+(2-\alpha)\\
(2-\alpha)(\alpha^2-1)+(2-2\alpha)\alpha+(3-2\alpha)\\
(1-\alpha)(\alpha^2-1)+(2-\alpha)\alpha+(2-2\alpha)
\end{pmatrix}
=0 .
\]
Let
\[
x^*\defas \frac{\tilde x}{1+\alpha+\alpha^2},
\qquad
y^*\defas \frac{\tilde y}{\tilde y_1+\tilde y_2+\tilde y_3}.
\]
The two identities above give
\[
x^{*\top}(A-\alpha Q)=0,
\qquad
(A-\alpha Q)y^*=0 .
\]
Hence, by \cref{lem:matrixgames}(v), $\val(A-\alpha Q)=0$, and therefore $\val W_0(\alpha)=0$. Since the game is strictly absorbing, \cref{lem:uniformlimit} implies that the undiscounted value of $\Gamma_\alpha$ is $\alpha$.

As in the previous example, the determinant calculation is not needed for the value, but one also has
\[
\det(A-tQ)=-(t^3-t-1),
\]
so the same polynomial $t^3-t-1$ appears in the determinant.

\section{Discussion}

\paragraph{Expressive power versus computational complexity.}
\cref{cor:char} measures the \emph{expressive power} of absorbing games: the undiscounted values of rational $m\times m$ absorbing games are exactly the real algebraic numbers of degree at most $m$. As $m$ grows, these values become arbitrarily rich, comprising irrational numbers of every algebraic degree.
This expressiveness does not come at a computational cost. By \citet{OB21}, the undiscounted value of a stochastic game is computed exactly in time polynomial in the number of pure stationary strategies of the two players. An $m\times n$ absorbing game has a single non-absorbing state, so its pure stationary strategies number $m$ and $n$, respectively; hence its undiscounted value, though possibly an irrational algebraic number of degree as large as $\min(m, n)$, is computed in time polynomial in $m$ and $n$. In contrast, for a general stochastic game the number of pure stationary strategies is exponential in the number of states, and whether the value is computable in polynomial time is a long-standing open problem.

\paragraph{Stochastic versus deterministic transitions.}
Our games use genuinely stochastic transitions. Choosing $N>\max_{ij}Q_{ij}$ in \cref{thm:main} makes every $q_{ij}\in(0,1)$, so each algebraic value is realized by a halting, recursive game in which absorption is probabilistic at every action pair. This randomness is essential. If a halting game is instead \emph{deterministic}, with $q_{ij}\in\{0,1\}$, then strict absorption forces $q_{ij}=1$ for every $(i,j)$, i.e., play absorbs after a single stage, and the game reduces to the one-shot matrix game with payoff matrix $w$, whose value is rational whenever the game is rational. Thus halting deterministic games realize only rational values, whereas \cref{thm:main} realizes every real algebraic number as the value of a halting recursive game with stochastic transitions.
The deterministic case is more delicate and remains largely open. By \citet[Theorem~1]{OV23}, deterministic absorbing games with $\min(m,n)<3$ obey the orderfield property, yet \citet{OV23} also exhibit a deterministic $3\times3$ game whose value is irrational, so determinism neither forces rationality nor is understood in general. Characterizing the values realized by deterministic absorbing games, and in particular by deterministic \emph{recursive} games, in which payoff accrues only upon absorption, is a natural question left open by our work.

\paragraph{Beyond minimal polynomials.}
The construction of \crefrange{sec:algebra}{sec:proof} and its conclusion that the undiscounted value equals $\alpha$ remain valid whenever $\mu\in\Q[z]$ is a monic polynomial of degree $m$ and $\alpha$ is a \emph{simple real root} of $\mu$ (the simplicity used in \cref{eq:simple}): every simple real root of a rational polynomial of degree $m$ is the undiscounted value of a rational $m\times m$ strictly absorbing and recursive game, which is positive recursive when $\alpha>0$. By \cref{cor:char} this yields no new values, but it shows that the construction depends neither on irreducibility nor on any field-theoretic property of $\Q[z]/(\mu)$, only on $\alpha$ being a simple root.

\paragraph{Explicit optimal strategies.}
The optimal stationary strategies of $\Gamma_\alpha$ for the undiscounted value are completely explicit and have a pleasant interpretation in terms of the arithmetic of $\alpha$: Player~1 plays the actions $k\in\I$ proportionally to the values $p_k(\alpha)$ of the basis polynomials at $\alpha$, and Player~2 plays the actions $k\in\J$ proportionally to the coordinates of the Lagrange interpolation polynomial $\ell$ in that basis. 
Indeed, let $x^*$ be the stationary strategy of Player~1 constructed in the proof of \cref{thm:main}, and let $\tau$ be any strategy of Player~2. Since $x^{*\top}W_0(\alpha)=0$, at every history before absorption the conditional expectation of $q_{IJ}(w_{IJ}-\alpha)$ is zero. Thus, writing $T$ for the absorption stage, $\E_{x^*,\tau}[(w(I_T,J_T)-\alpha)\mathbf 1_{\{T=\ell\}}]=0$ for every $\ell\geq1$. Strict absorption implies $T<\infty$ almost surely, and because $g=0$ the long-run average payoff equals $w(I_T,J_T)$; hence $\gamma(x^*,\tau)=\alpha$. The argument for the stationary strategy $y^*$ of Player~2, using $W_0(\alpha)y^*=0$, is symmetric.

\ACKNOWLEDGMENT{%
We thank Miquel Oliu-Barton and Guillaume Vigeral for bringing this problem to our attention and for helpful discussions at an early stage of this work.
The research was partially supported by the Austrian Science Fund (FWF) 10.55776/COE12 and by the ERC CoG 863818 (ForM-SMArt) grant.%
}

\bibliographystyle{plainnat}
\bibliography{references}

@inproceedings{AHK98,
  author    = {de Alfaro, Luca and Henzinger, Thomas A. and Kupferman, Orna},
  title     = {Concurrent reachability games},
  booktitle = {Proceedings of the 39th Annual Symposium on Foundations of Computer Science (FOCS)},
  pages     = {564--575},
  publisher = {IEEE Computer Society},
  year      = {1998},
}

@article{AO19,
  author  = {Attia, Luc and Oliu-Barton, Miquel},
  title   = {A formula for the value of a stochastic game},
  journal = {Proceedings of the National Academy of Sciences},
  volume  = {116},
  number  = {52},
  pages   = {26435--26443},
  year    = {2019},
}

@article{BK76,
  author  = {Bewley, Truman and Kohlberg, Elon},
  title   = {The asymptotic theory of stochastic games},
  journal = {Mathematics of Operations Research},
  volume  = {1},
  pages   = {197--208},
  year    = {1976},
}

@article{BF68,
  author  = {Blackwell, David and Ferguson, Thomas S.},
  title   = {The Big Match},
  journal = {Annals of Mathematical Statistics},
  volume  = {39},
  pages   = {159--163},
  year    = {1968},
}

@incollection{Eve57,
  author    = {Everett, Hugh},
  title     = {Recursive games},
  booktitle = {Contributions to the Theory of Games, III},
  editor    = {Dresher, Melvin and Tucker, Albert W. and Wolfe, Philip},
  series    = {Annals of Mathematics Studies},
  volume    = {39},
  pages     = {47--78},
  publisher = {Princeton University Press},
  year      = {1957},
}

@incollection{Gil57,
  author    = {Gillette, Dean},
  title     = {Stochastic games with zero stop probabilities},
  booktitle = {Contributions to the Theory of Games, III},
  editor    = {Dresher, Melvin and Tucker, Albert W. and Wolfe, Philip},
  series    = {Annals of Mathematics Studies},
  volume    = {39},
  pages     = {179--187},
  publisher = {Princeton University Press},
  year      = {1957},
}

@article{Koh74,
  author = {Elon Kohlberg},
  journal = {The Annals of Statistics},
  number = {4},
  pages = {724--738},
  publisher = {Institute of Mathematical Statistics},
  title = {Repeated Games with Absorbing States},
  volume = {2},
  year = {1974}
}

@article{MN81,
  author  = {Mertens, Jean-Fran\c{c}ois and Neyman, Abraham},
  title   = {Stochastic games},
  journal = {International Journal of Game Theory},
  volume  = {10},
  pages   = {53--66},
  year    = {1981},
}

@article{OB21,
  author  = {Oliu-Barton, Miquel},
  title   = {New algorithms for solving zero-sum stochastic games},
  journal = {Mathematics of Operations Research},
  volume  = {46},
  number  = {1},
  pages   = {255--267},
  year    = {2021},
}

@misc{OB23arxiv,
  author = {Oliu-Barton, Miquel},
  title  = {Absorbing games with irrational values},
  note   = {Preprint, arXiv:2307.03570},
  year   = {2023},
}

@article{OV23,
  author       = {Miquel Oliu{-}Barton and
                  Guillaume Vigeral},
  title        = {Absorbing games with irrational values},
  journal = {Operations Research Letters},
  volume       = {51},
  number       = {6},
  pages        = {555--559},
  year         = {2023}
}

@article{Sha53,
  author  = {Shapley, Lloyd S.},
  title   = {Stochastic games},
  journal = {Proceedings of the National Academy of Sciences of the United States of America},
  volume  = {39},
  pages   = {1095--1100},
  year    = {1953},
}

@article{vN28,
  author  = {John von Neumann},
  title   = {Zur Theorie der Gesellschaftsspiele},
  journal = {Mathematische Annalen},
  volume  = {100},
  pages   = {295--320},
  year    = {1928},
}

\end{document}